\documentclass[12pt]{amsart}

\usepackage[pdfauthor   = {Sebastian\ Posur},
            pdftitle    = {},
            pdfsubject  = {},
            pdfkeywords = {BGG correspondence;\ equivariant vector bundle;\ cohomology table},
            bookmarks=true,
            bookmarksopen=true,
            pagebackref=true,
            hyperindex=true,
            colorlinks=true,
            linkcolor=blue,
            citecolor=blue,
            filecolor=blue,
            urlcolor=blue,
            ]{hyperref}

\usepackage[utf8]{inputenc} 
\usepackage[T1]{fontenc}
\usepackage{a4wide}
\usepackage{lmodern}
\usepackage[english]{babel}
\usepackage{mathrsfs}
\usepackage{etex}
\usepackage{mathtools}
\usepackage{latexsym}
\usepackage{amssymb}
\usepackage{amsthm}
\usepackage{amsmath}
\usepackage{caption}
\usepackage{mathabx}
\usepackage{commath}
\usepackage{colortbl}

\usepackage[all]{xy}
\usepackage{verbatim}
\usepackage{listings}
\usepackage{fancyvrb}
\usepackage{calc}
\usepackage{graphicx}

\usepackage[dvipsnames]{xcolor}
\usepackage{accents} 
\usepackage{enumerate}
\usepackage{wrapfig}
\usepackage{tikz}
\usepackage{tikz-cd}
\usetikzlibrary{automata,shapes,arrows,matrix,backgrounds,positioning,plotmarks,calc,patterns,matrix,decorations.pathreplacing,decorations.pathmorphing,decorations.text,decorations.markings}
\usepackage[colorinlistoftodos,shadow]{todonotes}

\usepackage{multirow}
\usepackage{mdwlist}

\usepackage{stmaryrd}
\usepackage{mathdots} 

\usepackage{fancyvrb}

\usepackage[toc,page]{appendix}
\usepackage{float}

\usepackage{extarrows}
\usepackage{pdflscape}
\usepackage{rotating}
\usepackage{etex}
\newtheoremstyle{mytheoremstyle} 
    {5pt}                    
    {5pt}                    
    {\itshape}                   
    {\parindent}                           
    {\bf}                   
    {.}                          
    {.5em}                       
    {}  

\theoremstyle{mytheoremstyle}

\newtheorem{theorem}{Theorem}[section]

\newtheorem{lemma}[theorem]{Lemma}

\newtheorem{corollary}[theorem]{Corollary}

\newtheoremstyle{mytdefintionstyle} 
    {5pt}                    
    {5pt}                    
    {\rm}                   
    {\parindent}                           
    {\bf}                   
    {.}                          
    {.5em}                       
    {}  

\theoremstyle{remark}
\newtheorem{remark}[theorem]{Remark}

\theoremstyle{mytdefintionstyle}
\newtheorem{definition}[theorem]{Definition}

\newtheorem{example}[theorem]{Example}

\newtheorem{construction}[theorem]{Construction}

\newtheorem*{notationnonumber}{Notation}

\newtheoremstyle{exmp_contd} 
{\topsep} {\topsep}%
{\upshape}
{}
{\bfseries}
{}
{ }
{\thmname{#1}\,\thmnumber{ #2}\thmnote{#3}\enspace(continued)}

\theoremstyle{exmp_contd}

\usepackage{xspace}
\definecolor{ExQ}{HTML}{0000FF}
\definecolor{Dec}{HTML}{E07B00}

\newcommand{\CapPkg}{\textsc{Cap}\xspace}

\DeclareMathOperator{\Proj}{Proj}

\newcommand{\AC}{\mathbf{A}}

\newcommand{\CC}{\mathbf{C}}

\newcommand{\TC}{\mathbf{T}}
\newcommand{\UC}{\mathbf{U}}

\newcommand{\N}{\mathbb{N}}
\newcommand{\Z}{\mathbb{Z}}

\newcommand{\Q}{\mathbb{Q}}

\newcommand{\C}{\mathbb{C}}

\newcommand{\Pro}{\mathbb{P}}

\newcommand{\CH}{\mathrm{H}}
\newcommand{\Syz}{\mathrm{Syz}}

\newcommand{\Coh}{\mathfrak{Coh}\,}

\newcommand{\CTate}{\mathfrak{Tate}}

\newcommand{\rank}{\mathrm{rank}}
\newcommand{\chern}{\mathrm{chern}}
\newcommand{\BGG}{\mathrm{BGG}}
\newcommand{\TateRes}{\mathrm{Tate}}

\newcommand{\Rep}{\mathrm{Rep}_{k}}

\newcommand{\Obj}{\mathrm{Obj}}

\newcommand{\im}{\mathrm{im}}

\DeclareMathOperator{\Irr}{\mathrm{Irr}}
\newcommand{\id}{\mathrm{id}}

\DeclareMathOperator{\Hom}{\mathrm{Hom}}

\DeclareMathOperator{\kernel}{\mathrm{ker}}

\DeclareMathOperator{\cokernel}{\mathrm{coker}}

\DeclareMathOperator{\image}{\mathrm{im}}

\DeclareMathOperator{\soc}{\mathrm{soc}}
\DeclareMathOperator{\sgn}{\mathrm{sgn}}
\DeclareMathOperator{\rad}{\mathrm{rad}}

\DeclareMathOperator{\Sym}{\mathrm{Sym}}

\DeclareMathOperator{\Ext}{\mathrm{Ext}}

\newcommand{\HPo}{\mathrm{HP}}

\newcommand{\kvec}{k\mathrm{\text{-}vec}}

\newcommand{\grmod}{\text{-}\mathrm{grmod}}
\newcommand{\stgrmod}[1]{\underline{{#1}\text{-}\smash{\mathrm{grmod}}}}

\newcommand{\ModLeft}{\mathrm{\text{-}mod}}

\newcommand{\GAP}{\texttt{GAP}\xspace}

\newcommand{\bfindex}[1]{{\textbf{#1}}{\index{{#1}}}}

\tikzset{round left paren/.style={ncbar=0.5cm,out=120,in=-120}}
\tikzset{round right paren/.style={ncbar=0.5cm,out=60,in=-60}}
\newcolumntype{C}[1]{>{\centering\arraybackslash$}p{#1}<{$}}
\newlength{\mycolwd}
\usepackage{array, xcolor}
\definecolor{lightgray}{gray}{0.8}
\newcolumntype{L}{>{\raggedleft}p{0.28\textwidth}}
\newcolumntype{R}{p{0.8\textwidth}}

\definecolor{ctcolor}{gray}{0.95}
\newcommand\ct{\cellcolor{ctcolor}}
\definecolor{ctucolor}{gray}{0.85}
\newcommand\ctu{\cellcolor{ctucolor} ?}
\newcommand\ctuc{\cellcolor{ctucolor}}

\makeatletter
\newcommand{\thickhline}{%
    \noalign {\ifnum 0=`}\fi \hrule height 1pt
    \futurelet \reserved@a \@xhline
}
\newcolumntype{"}{@{\hskip\tabcolsep\vrule width 1pt\hskip\tabcolsep}}
\makeatother

\author{Sebastian Posur}
\address{Department of mathematics, University of Siegen, 57068 Siegen, Germany}
\email{\href{mailto:Sebastian Posur <sebastian.posur@uni-siegen.de>}{sebastian.posur@uni-siegen.de}}

\begin{document}

\title[Constructing equivariant vector bundles via the BGG correspondence]{Constructing equivariant vector bundles via the BGG correspondence}

\begin{abstract}
We describe a strategy for the construction of 
finitely generated $G$-equivariant $\Z$-graded modules $M$
over the exterior algebra for a finite group $G$.
By an equivariant version of the BGG correspondence, $M$
defines an object $\mathcal{F}$ in the bounded derived category
of $G$-equivariant coherent sheaves on projective space.
We develop a necessary condition for $\mathcal{F}$ being isomorphic to a vector bundle
that can be simply read off from the Hilbert series of $M$. 
Combining this necessary condition with the computation of
finite excerpts of the cohomology table of $\mathcal{F}$
makes it possible to enlist a class of
equivariant vector bundles on $\Pro^4$
that we call strongly determined
in the case where $G$ is the alternating group on $5$ points.
\end{abstract}

\maketitle

\tableofcontents

\section{Introduction}

Interesting problems concerning vector bundles on projective space
have remained open for many years.
For example, Hartshorne conjectures in \cite{Har74}
that there are no indecomposable vector bundles of rank $2$ on $\Pro_{\C}^n$
for $n \geq 7$, which is equivalent to
the conjecture that every codimension $2$ nonsingular subvariety of $\Pro_{\C}^n$
is a complete intersection for $n \geq 7$.
Furthermore, the only known rank $2$ indecomposable vector bundle on $\Pro_{\C}^4$
is the famous Horrocks-Mumford bundle \cite{HM73} (up to twists by line bundles
and pullbacks by finite maps \cite{DS87}). It is unclear whether there exist more such vector bundles.

The purpose of this paper is 
to provide tools for the construction of $G$-equivariant vector bundles on projective space for a finite group $G$
over a splitting field $k$.
The main idea is to
start with a triple of $\Z$-graded \emph{irreducible} $k$-representations $V, T, U$ 
of $G$ concentrated in degrees $-1, t,u \in \Z$ (where we assume $t \not = u$),
and to consider $G$-equivariant $E$-linear graded maps of the form
\[
 \widehat{\phi}: E \otimes_k (\bigwedge^{n+1}W \otimes_k T) \longrightarrow E \otimes_k (\bigwedge^{n+1}W \otimes_k U),
\]
where $\dim_k( V ) = n+1$,
$E = \bigwedge V$ is the graded exterior algebra of $V$, and $W := \Hom_k(V, k)$ the graded $k$-dual.
Due to an equivariant version of the famous BGG correspondence,
the module $M = \kernel\widehat{\phi}$ corresponds to an object $\mathcal{F}$ in the bounded derived category of $G$-equivariant coherent sheaves on $\Pro V$
with some hypercohomology groups prescribed by $T$ and $U$.
The tools described in this paper will then help us with the decision whether $\mathcal{F}$ is isomorphic to a vector bundle.

The paper is organized as follows.
Section \ref{section:BGG} gives the theoretical background of our construction strategy.
We start with a review of the classical BGG correspondence in Subsection \ref{subsection:classical_bgg}
and show how it can be used for the computation of hypercohomology groups (Subsection \ref{subsection:hypercohomology}).
In Subsection \ref{subsection:construction_strategy}
we state our construction strategy of objects $\mathcal{F}$ in $D^b( \Coh \Pro V )$ in the 
non-equivariant case, where $\Coh\Pro V$ is the category of coherent sheaves on $\Pro V$.
The next two subsections are concerned with the question when $\mathcal{F}$
is isomorphic to a vector bundle.
This question can in theory be decided by finding two particularly sparse columns in the
cohomology table of $\mathcal{F}$ (Lemma \ref{lemma:sufficient_condition}),
but since we cannot in practice compute the entire cohomology table 
for finding (or disproving the existence) of two such columns,
we also give a necessary condition in Corollary \ref{corollary:necessary_condition}
that can be simply read off from the Hilbert series of $M = \kernel\widehat{\phi}$
(and thus is computationally quite cheap).
Subsection \ref{subsection:equivariant} and Subsection \ref{subsection:adaption}
lift the BGG correspondence and our construction strategy to the $G$-equivariant setup.

In Section \ref{section:A5} we apply our tools
in the case where $G$ is the alternating group $A_5$ on $5$ points
and $V$ is its unique irreducible representation of dimension $5$.
We construct all those $G$-equivariant vector bundles 
for which the corresponding map $\widehat{\phi}$
is already determined (up to non-zero scalar)
by its source and range. We call such vector bundles \textbf{strongly determined}.
Interestingly, in the case $G = A_5$, all such bundles turn out to have supernatural cohomology,
a notion of importance in Boij-Söderberg theory \cite{ES09}.
We summarize our findings in Theorem \ref{theorem:A5_main_theorem}.

All necessary computations for obtaining Theorem \ref{theorem:A5_main_theorem}
were performed with \CapPkg \cite{CAP-project},
a software project for category theory implemented in \textsf{GAP} \cite{GAP4}.
We address our strategy to compute in a $G$-equivariant $E$-linear
setup on the computer in Appendix \ref{appendix}.

\begin{notationnonumber}
 By a module we will always mean a left module.
 A graded module is understood as a $\Z$-graded module over a $\Z$-graded ring.
 Given a graded module $M = \bigoplus_{i \in \Z}M_i$, its \textbf{twist}
 by $d \in \Z$ is denoted by $M(d) = \bigoplus_{i \in \Z}M_{d+i}$.
 Similarly, the twist by $d$ of a coherent sheaf $\mathcal{F}$
 (or more generally of a cochain complex of coherent sheaves)
 on projective space
 is denoted by $\mathcal{F}(d) = \mathcal{F} \otimes \mathcal{O}(d)$.
 If we consider an object $T$ in a triangulated category $\mathcal{T}$,
 its \textbf{translation} by $d$ is denoted by $T[d]$.
 If we say a property holds for $\mathcal{F}$ \textbf{up to translation and twist}, e.g., being isomorphic to a vector bundle
 with Chern polynomial given by $1+2h^2+2h^3$,
 we mean that this property is true for at least one of the objects in $\{ \mathcal{F}(d)[e] \mid d,e \in \Z \}$.
\end{notationnonumber}

\section{Constructing equivariant vector bundles via the BGG correspondence}\label{section:BGG}

\subsection{The BGG correspondence}\label{subsection:classical_bgg}

Let $k$ be a field, $n \in \N_0$, and $V$ be an $(n+1)$-dimensional
$\Z$-graded $k$-vector space concentrated in degree $-1$. We denote by $W$ the graded $k$-dual
of $V$ which is necessarily concentrated in degree $1$.
We write $E := \bigwedge V$ for the graded exterior algebra of $V$,
and $S := \Sym( W )$ for the graded symmetric algebra of $W$.
Let $E\grmod$ denote the abelian category of finitely generated graded $E$-modules.
Its corresponding stable category $\stgrmod{E}$ 
is the quotient category of $E\grmod$ modulo those morphisms that factor
over a projective object. 
Since $E\grmod$ is a Frobenius category, i.e., it has enough projectives and injectives,
and these two classes of objects coincide, $\stgrmod{E}$ can be equipped with a triangulated structure \cite[Theorem 2.9]{GMAlgebraV}.
If we regard a graded $E$-module $M$ as an object
in $\stgrmod{E}$, we shall write $\underline{M}$.
Furthermore, we denote the category of coherent sheaves on projective space $\Pro V := \Proj( S )$
by $\Coh \Pro V$, and its bounded derived category by $D^b( \Coh \Pro V )$.
\begin{theorem}[BGG]
 There is an equivalence of triangulated categories
 \[
  \BGG: \stgrmod{E} \xrightarrow{\sim} D^b( \Coh \Pro V )
 \]
 that maps $\underline{k(i)}$ to $\mathcal{O}_{\Pro V}(i)[-i]$ for all $i \in \Z$.
 Here, we regard $k \simeq E/ \rad(E)$ as a graded $E$-module concentrated in degree $0$.
\end{theorem}

The BGG correspondence was first proved in \cite{BGG}.
A more detailed construction and proof of its correctness can be found in \cite[Appendix]{OSS}.

\subsection{Tate sequences and hypercohomology}\label{subsection:hypercohomology}

The BGG correspondence provides a great computational tool for sheaf cohomology (Theorem \ref{theorem:BGG_and_cohomology})
that we are now going to explain. We will use the more general language of Frobenius categories
and stable categories in the sense of \cite{GMAlgebraV} (since this language also applies to the equivariant setup discussed in Subsection \ref{subsection:equivariant}),
but it is safe to simply think of $E\grmod$ and $\stgrmod{E}$.

\begin{definition}
 Let $\AC$ be a Frobenius category
 and let $K^{\bullet}(\AC)$ denote its category of cochain complexes with cochain homotopy classes
 as morphisms. The \textbf{category of Tate sequences} $\CTate( \AC )$
 is defined as the full subcategory of $K^{\bullet}(\AC)$
 generated by (possibly infinite) acyclic cochain complexes only having projective objects
 in each cohomological degree. 
\end{definition}

Note that $\CTate( \AC )$ inherits a triangulated structure from $K^{\bullet}(\AC)$.
Moreover, sending a cochain complex in $\CTate(\AC)$ to its $0$-th syzygy object
yields a functor 
\[
\Syz: \CTate( \AC ) \rightarrow \underline{\AC}
\]
mapping into the stable category of $\AC$.
We also get a functor the other way around
using the following construction.
\begin{construction}\label{construction:Tate_res}
 Let $\underline{A} \in \underline{\AC}$. Take a projective resolution $P^{\bullet} \twoheadrightarrow A$
 and an injective resolution $P \hookrightarrow I^{\bullet}$. Joining both resolutions with the differential
 \[
  P^0 \twoheadrightarrow A \hookrightarrow I^0
 \]
 (having cohomological degree $0$)
 yields an acyclic cochain complex of projectives, thus a Tate sequence, denoted by $\TateRes( \underline{A} )$.
\end{construction}

 See \cite[Chapter 5, 2.9.1]{GMAlgebraV} for the following folklore theorem.
\begin{theorem}\label{theorem:stable_cat}
 Let $\AC$ be a Frobenius category.
 Then the two constructions $\Syz$ and $\TateRes$ yield an equivalence of triangulated categories
 \[
   \CTate( \AC ) \simeq \underline{\AC}.
 \]
\end{theorem}

Now, let $\mathcal{F} \in D^b( \Coh \Pro V )$
correspond to $\underline{M} \in \stgrmod{E}$ via the BGG correspondence. It is easy to
see that $M$ can be chosen as a reduced $E$-module, i.e., it
does not contain a non-zero projective submodule.
If we choose a minimal projective and injective resolution in Construction \ref{construction:Tate_res},
then $\TateRes(\underline{M})$ has the property that 
tensoring with the $E-E$ bimodule $k$ yields zero differentials in all cohomological degrees.
We call such Tate sequences \textbf{minimal} and denote the minimal Tate sequence of $\mathcal{F}$ by $\TateRes( \mathcal{F} )$.

\begin{theorem}\label{theorem:BGG_and_cohomology}
 Let $\omega_E$ be the graded $k$-dual of $E$.
 The minimal Tate sequence of an object
 $\mathcal{F} \in D^b( \Coh \Pro V )$ is of the form
\begin{center}
          \begin{tikzpicture}[transform shape,mylabel/.style={thick, draw=black, align=center, minimum width=0.5cm, minimum height=0.5cm,fill=white}]
            \coordinate (r) at (3.5,0);
            \coordinate (d) at (0,-0.75);
            \node (A1) {$\dots$};
            \node (A2) at ($(A1) + (r)$) {$\omega_E \otimes_{k} \big(\bigoplus_{i=0}^n \CH^i( \mathcal{F}(-i) ) \big)$};
            \node (A3) at ($(A2) + 2*(r)$) {$\omega_E \otimes_{k} \big(\bigoplus_{i=0}^n \CH^i( \mathcal{F}(-i + 1) ) \big)$};
            \node (A4) at ($(A3) + 1.1*(r)$) {$\dots$};
            \draw[->,thick] (A2) --node[above]{$d^0$} (A3);
            \draw[->,thick] (A1) -- (A2);
            \draw[->,thick] (A3) -- (A4);
          \end{tikzpicture}
  \end{center}
  where the hypercohomology groups $\CH^i( \mathcal{F}(j) )$ are regarded as graded $k$-vector spaces concentrated in degree $j$.
  Thus, from this cochain complex, we can read off the hypercohomology groups of $\mathcal{F}$ as the socles of the objects.
\end{theorem}
\begin{proof}
 If $\mathcal{F}$ is a coherent sheaf, then this is \cite[Theorem 4.1]{EFS}.
 We present an alternative proof which directly uses the BGG correspondence and works for hypercohomology without extra effort:
 \begin{align*}
  \CH^{i}( \mathcal{F} ) &\simeq \Hom_{D^b( \Coh \Pro V )}( \mathcal{O}_{\Pro V}, \mathcal{F}[i] )\\
  &\simeq \Hom_{ \stgrmod{E}}\left( k, \Syz(   \TateRes( \mathcal{F} )[i])\right)
 \end{align*}
 Now, since we chose $\TateRes( \mathcal{F} )$ as a minimal Tate sequence, 
 its syzygy objects are reduced $E$-modules. From this,
 we conclude
 \begin{align*}
 \Hom_{ \stgrmod{E}}\left( k, \Syz(   \TateRes( \mathcal{F} )[i])\right)
  &\simeq \Hom_{ E\grmod}( k, \Syz(   \TateRes( \mathcal{F} )[i])) \\
  &\simeq \soc(\TateRes(\mathcal{F})^{i})_0
 \end{align*}
 where we use that elements in the socle of an $E$-module $M$
 correspond to $E$-module homomorphisms $k \rightarrow M$.
 Substituting $\mathcal{F}(j)$ for $\mathcal{F}$ gives the result in the general case.
\end{proof}

\begin{remark}\label{remark:in_free_terms}
 Using the natural isomorphism $\omega_E \simeq E \otimes_k \bigwedge^{n+1}W$ we can rewrite
 the objects of the Tate sequence in Theorem \ref{theorem:BGG_and_cohomology}
 in terms of free $E$-modules:
 \[
  \omega_E \otimes_{k} \big(\bigoplus_{i=0}^n \CH^i( \mathcal{F}(-i) ) \big) \simeq E \otimes_k \big(\bigwedge^{n+1}W \otimes_{k} \bigoplus_{i=0}^n \CH^i( \mathcal{F}(-i) ) \big).
 \]
\end{remark}

\subsection{The construction strategy}\label{subsection:construction_strategy}

The BGG correspondence makes it possible to construct objects
in $D^b( \Coh \Pro V )$ with prescribed hypercohomology groups.
Let $T,U$ be finite dimensional $\Z$-graded $k$-vector spaces
concentrated in degrees $t$, $u$, respectively.
Assume $t \not= u$.
Let
\[
 \phi: (\bigwedge^{n+1}W \otimes_k T) \longrightarrow E \otimes_k (\bigwedge^{n+1}W \otimes_k U)
\]
be a monomorphism of graded vector spaces.
By the universal property of free modules, $\phi$ gives rise to
a morphism of graded $E$-modules
\[
 \widehat{\phi}: E \otimes_k (\bigwedge^{n+1}W \otimes_k T) \longrightarrow E \otimes_k (\bigwedge^{n+1}W \otimes_k U)
\]
whose kernel is a reduced $E$-module
(it does not contain a non-zero projective submodule).
The natural monomorphism
\[
 \kernel( \widehat{\phi} ) \hookrightarrow E \otimes_k (\bigwedge^{n+1}W \otimes_k T)
\]
is an injective hull (since $t \not= u$).
Thus, $E \otimes_k (\bigwedge^{n+1}W \otimes_k T)$ is the first object
in a minimal injective resolution of $\kernel( \widehat{\phi} )$.
From Theorem \ref{theorem:BGG_and_cohomology} and Remark \ref{remark:in_free_terms}
we conclude 
\begin{equation}\label{equation:h0}
H^i\left(\BGG( \kernel( \widehat{\phi} ) )(-i+1) \right)
=
\left\{ 
    \begin{array}{cc}
                 T, & \text{for~}(-i + 1) = t, \\
                 0, & \text{else,}
    \end{array} 
   \right. 
\end{equation}
for all $i \in \Z$.
Since $\widehat{\phi}$ does not have to be the next step of a minimal injective
resolution, we cannot give an interpretation for $U$ in terms of hypercohomology groups as well.
However, we will improve this situation when we pass to an equivariant setup in Subsection \ref{subsection:adaption}.

In conclusion, choosing a monomorphism $\phi: (\bigwedge^{n+1}W \otimes_k T) \longrightarrow E \otimes_k (\bigwedge^{n+1}W \otimes_k U)$
defines an object in $D^b( \Coh \Pro V )$ with hypercohomology as shown above.
But since our ultimate goal is the construction of vector bundles,
we have to learn how to detect if $\BGG( \kernel( \widehat{\phi} ) )$
is isomorphic to a vector bundle.
\subsection{Cohomology tables of vector bundles}

Let $\mathcal{F} \in D^b( \Coh \Pro V )$.
To simplify notation, we set $h^i_j := \dim_k\CH^i(\mathcal{F}(j))$
and refer to $\left( h^i_j\right)_{i,j}$ as the \textbf{cohomology table of $\mathcal{F}$}
(even though the table contains the dimensions of hypercohomology groups).
We adopt the convention of \cite{EFS} and display a cohomology table
as follows:

\settowidth{\mycolwd}{$h^{n}_{n-1 + 1}$}
\[
 \begin{array}{c*{5}{C{\mycolwd}}}   
     & \ct\vdots & \vdots& \ct\vdots & \\
     & \ct& & \ct&\\[-1em]
     \dots & \ct h^2_{-3} & h^2_{-2}& \ct h^2_{-1} & \dots \\
     & \ct& & \ct&\\[-1em]
     \dots & \ct h^1_{-2}& h^1_{-1}& \ct h^1_{0}& \dots \\
     & \ct& & \ct&\\[-1em]
     \dots & \ct h^0_{-1}& h^0_{0}& \ct h^{0}_{1}& \dots \\
     & \ct& & \ct&\\[-1em]
     & \ct\vdots & \vdots& \ct\vdots & \\
 \end{array}
\]

From its cohomology table, lots of information about $\mathcal{F}$ may be deduced.
For example, the following statements are well-known:
\begin{itemize}
 \item $\mathcal{F}$ is isomorphic to a coherent sheaf if and only if
       there exists a $u_0 \in \N$ such that
       $h^i_j = 0$ for all $j > u_0$ and $i \not= 0$.
 \item $\mathcal{F}$ is isomorphic a vector bundle if and only if
       it is isomorphic to a coherent sheaf and there exists an $l_0 \in \N$ such that
       $h^i_j = 0$ for all $j < -l_0$ and $i \not= n$.
\end{itemize}
Unfortunately, these statements assume that we know the whole cohomology table,
whereas concrete computations usually give us only excerpts bounded to the left and right.
Luckily, it is true that if such an excerpt of a cohomology table of $\mathcal{F}$
``looks like'' the cohomology table of a vector bundle,
then $\mathcal{F}$ is already a vector bundle.
The following lemma makes this statement precise.
\begin{lemma}[Vector bundle criterion]\label{lemma:sufficient_condition}
 Let $\mathcal{F} \in D^b( \Coh \Pro V )$.
 Then $\mathcal{F}$ is isomorphic to a vector bundle
 if and only if
 its cohomology table
 has two columns with only one non-zero entry such that the 
 non-zero entry in the left column is in the $n$-th row and the non-zero entry in the right column is in the $0$-th row.
 \settowidth{\mycolwd}{$h^{n}_{n-1 + 1}$}
\[
 \begin{array}{c*{5}{C{\mycolwd}}}   
      & \ct\vdots & \vdots& \ct      & \vdots& \ct \vdots\\
      & \ct 0     & ?     & \ct\dots & ?     & \ct 0     \\
 & \ct h^n_l & ?     & \ct\dots & ?     & \ct 0     \\
      & \ct 0     & ?     & \ct\dots & ?     & \ct 0     \\
& \ct \vdots& \vdots& \ct      & \vdots& \ct \vdots\\
      & \ct 0     & ?     & \ct\dots & ?     & \ct 0     \\
& \ct 0     & ?     & \ct\dots & ?     & \ct h^0_r \\
      & \ct 0     & ?     & \ct\dots & ?     & \ct 0     \\
      & \ct\vdots & \vdots& \ct      & \vdots& \ct\vdots \\
 \end{array}
\]
\end{lemma}
\begin{proof}
 We only have to show the ``if'' direction, which
 can be deduced from \cite[Lemma 7.4]{EFS}.
 The reasoning is as follows.
 Given a cohomology table having an $r$-th column with only one non-zero entry $h\not=0$,
 then the non-zero entries to its right can only lie in the
 dark gray range indicated in the following picture:
 \begin{center}
          \begin{tikzpicture}[transform shape,mylabel/.style={thick, draw=black, align=center, minimum width=0.5cm, minimum height=0.5cm,fill=white}]
            \coordinate (r) at (3.5,0);
            \coordinate (d) at (0,-0.75);
            \node (A1) 
            {$
              \begin{array}{C{\mycolwd}C{\mycolwd}C{\mycolwd}C{\mycolwd}}

              \ct\vdots & \vdots     & \vdots\ct   &         \\
              \ct0      & 0          & 0      \ct  & \cdots  \\
              \ct h     &\ctu        &\ctu         & \cdots\ctuc  \\
              \ct0      &\ctu        &\ctu         & \cdots\ctuc  \\
              \ct\vdots &\ctuc\vdots &\ctuc\vdots  &       \ctuc  \\
              \ct0      &\ctu        &\ctu         & \cdots\ctuc  \\
              \ct0      &0           &0        \ct & \cdots  \\
              \ct\vdots & \vdots     &\vdots\ct    &         \\
              \end{array}
            $};
            \node (B) at ($(A1) + (r) - 1.65*(d)$) {};
            \node (C) at ($(A1) + (r) + 1.65*(d)$) {};
            \draw[decorate,decoration=brace, thick] (B) -- node[right] {$n+1~$} (C);
          \end{tikzpicture}
  \end{center}
To see this, let $M$ denote the $r$-th syzygy object of $\TateRes( \mathcal{F} )$.
Then 
\[M \hookrightarrow \TateRes( \mathcal{F} )^r \rightarrow \TateRes( \mathcal{F} )^{r+1} \rightarrow \dots \]
is a minimal injective resolution from whose socles we can read off the right part of the cohomology table.
Since taking socles of graded $E$-modules is a functorial process equivalent to applying $\Hom_{E}(k,-)$,
we have $\soc( \TateRes( \mathcal{F} )^{r+i} ) =  \Ext^i_{E}(k,M)$ for $i \geq 0$.
But this $\Ext$-group can also be computed as a subquotient of $\Hom_{E}( P^i, M )$,
where $P^{\bullet}$ is a minimal projective resolution of $k$.
Now, the dark gray range simply consists of those degrees in which the objects of $\Hom_{E}( P^{\bullet}, M )$
may be supported.
By dualizing $M$, we have the following range for the non-zero entries to the left:
 \begin{center}
          \begin{tikzpicture}[transform shape,mylabel/.style={thick, draw=black, align=center, minimum width=0.5cm, minimum height=0.5cm,fill=white}]
            \coordinate (r) at (3.5,0);
            \coordinate (d) at (0,-0.75);
            \node (A1) 
            {$
              \begin{array}{C{\mycolwd}C{\mycolwd}C{\mycolwd}C{\mycolwd}}   
                &\ct\vdots & \vdots & \ct\vdots \\
                \cdots&\ct0 & 0 & \ct0 \\
               \ctuc \cdots&\ctu & \ctu & \ct0 \\
               \ctuc &\ctuc\vdots & \ctuc\vdots & \ct\vdots \\
               \ctuc \cdots&\ctu & \ctu & \ct0 \\
               \ctuc \cdots&\ctu & \ctu & \ct h \\
                \cdots&\ct0 & 0 & \ct0 \\
                &\ct\vdots & \vdots & \ct\vdots \\
              \end{array}
            $};
            \node (B) at ($(A1) - (r) - 1.65*(d)$) {};
            \node (C) at ($(A1) - (r) + 1.65*(d)$) {};
            \draw[decorate,decoration=brace, thick] (C) -- node[left] {$n+1~$} (B);
          \end{tikzpicture}
  \end{center}
  Applying these two restrictions to the two columns of our given cohomology table,
  we can conclude that the only non-zero entries lie in the following range:
 \begin{center}
          \begin{tikzpicture}[transform shape,mylabel/.style={thick, draw=black, align=center, minimum width=0.5cm, minimum height=0.5cm,fill=white}]
            \coordinate (r) at (3.5,0);
            \coordinate (d) at (0,-0.75);
            \node (A1) 
            {$
              \begin{array}{C{\mycolwd}C{\mycolwd}C{\mycolwd}C{\mycolwd}C{\mycolwd}C{\mycolwd}C{\mycolwd}C{\mycolwd}C{\mycolwd}}   
               & \vdots & \ct\vdots& \vdots & \ct& \vdots & \ct\vdots & \vdots & \\
          \cdots & 0 &\ct0 & 0 & \cdots\ct & 0 & \ct 0 & 0 & \cdots \\
          \cdots & \ctu & h^n_l\ct & \ctu & \cdots\ctuc & \ctu &\ct 0 & 0 & \cdots \\
          \cdots& 0 & \ct 0& \ctu &\ctuc \cdots & \ctu & \ct 0 & 0 & \cdots \\
          & \vdots &\ct\vdots & \ctuc \vdots& \ctuc & \ctuc\vdots & \ct \vdots & \vdots & \\
          \cdots& 0 & \ct 0& \ctu &\ctuc \cdots & \ctu & \ct 0 & 0 & \cdots \\
          \cdots& 0 & \ct 0& \ctu &\ctuc \cdots & \ctu & \ct h^0_r & \ctu & \cdots \\
          \cdots & 0 &\ct0 & 0 & \cdots\ct & 0 & \ct 0 & 0 & \cdots \\
          & \vdots & \ct\vdots& \vdots & \ct& \vdots & \ct\vdots & \vdots & \\
              \end{array}
            $};
          \end{tikzpicture}
  \end{center}
 
  Thus, $\mathcal{F}$ corresponds to a vector bundle.
\end{proof}
In \cite[Example 7.3]{EFS}, Eisenbud, Fl\o{}ystad, and Schreyer
use this criterion to prove that
a particular object in $\stgrmod{E}$ gives rise to the famous Horrocks-Mumford bundle.

\subsection{An isomorphism of Grothendieck groups}\label{subsection:K0}
In this subsection we are going to work with Chern classes
and so we assume that $k$ is an algebraically closed field.
We give a necessary condition for an object $\underline{M} \in \stgrmod{E}$
to correspond (up to translation and twist) to a vector bundle in $D^b( \Coh \Pro V )$
that can be simply read off from the Hilbert series of $M$
and thus is computationally much faster than computing (an excerpt of) the cohomology table.
Our main tools are the Grothendieck groups of abelian and triangulated categories.

First, let $\AC$ be an abelian category.
An additive invariant for $\AC$ consists of an abelian group $H$
and a map $\alpha: \Obj( \AC ) \rightarrow H$ such that
$\alpha( B ) = \alpha( A ) + \alpha( C )$ for all short exact sequences
$0 \rightarrow A \rightarrow B \rightarrow C \rightarrow 0$.
The Grothendieck group $K_0(\AC)$ is defined as the universal additive invariant of $\AC$.
It can be constructed as the abelian group freely
generated by the the isomorphism classes\footnote{To avoid set theoretic issues, we suppose that the isomorphism classes form a set.} $[A]$
for $A \in \AC$ modulo the relations $[A] - [B] + [C] = 0$ for all
short exact sequences as above. Any additive invariant $\alpha$ 
uniquely factors over the universal map $A \mapsto [A]$.

\begin{example}\label{example:K_0coh}
In the case $\AC = \Coh\Pro{V}$, 
it follows from the Koszul complex and the existence of resolutions by line bundles that 
\begin{equation}\label{equation:k0coh}
K_0(\Coh\Pro{V}) \simeq \Z[x,x^{-1}]/ \langle \sum_{i=0}^{n+1}(-1)^ix^{-i}\binom{n}{i}\rangle_{\Z[x,x^{-1}]} \simeq \langle x^0, \dots, x^n \rangle_{\Z}, 
\end{equation}
where the first isomorphism identifies the line bundles $\mathcal{O}_{\Pro{V}}(i)$ with the indeterminates $x^i$ for $i \in \Z$.
In this case, well-known additive invariants for $\Coh\Pro{V}$ are the rank
and the Chern polynomial of a coherent sheaf $\mathcal{F}$.
From the universal property of the Grothendieck group, they can be computed as follows:
If $[\mathcal{F}] = \sum_{i \in \Z} a_i[\mathcal{O}_{\Pro{V}}(i)]$ for finitely many non-zero $a_i \in \Z$,
then 
\[\rank(\mathcal{F}) = \sum_{i\in \Z} a_i \in \Z\]
and
\[
 \chern( \mathcal{F} ) = \prod_{i \in \Z} (1+ih)^{a_i} \in (\Z[h]/\langle h^{n+1}\rangle)^{\times}
\]
where $(\Z[h]/\langle h^{n+1}\rangle)^{\times}$ denotes the group of units of the ring $\Z[h]/\langle h^{n+1}\rangle$.
\end{example}
%

Next, we turn to the Grothendieck group $K_0(\TC)$ of a triangulated category $\TC$.
It is defined as the abelian group freely generated by the isomorphism classes $[T]$ 
for $T \in \TC$ modulo the relations $[S] - [T] + [U] = 0$ for
every distinguished triangle $S \rightarrow T \rightarrow U \rightarrow S[1]$.
Since any short exact sequence in an abelian category $\AC$
induces a distinguished triangle in $D^b( \AC )$, we have an inclusion
$K_0( \AC ) \hookrightarrow K_0( D^b( \AC ) )$
that actually is an isomorphism with
\[
 K_0( D^b( \AC ) ) \rightarrow K_0( \AC ): [ C^{\bullet} ] \mapsto \sum_{i \in \Z} (-1)^i[H^i(C^{\bullet})]
\]
as an inverse. In particular, we have an isomorphism $K_0( \Coh \Pro{V}  ) \simeq K_0( D^b( \Coh \Pro{V}  ) )$
that allows us to speak about the rank and the Chern polynomial of an arbitrary object in $D^b( \Coh \Pro{V}  )$.
Moreover, given a triangulated functor $F: \TC \rightarrow \UC$ between triangulated categories,
it functorially induces a well-defined group homomorphism
\[
 K_0(F): K_0( \TC ) \rightarrow K_0( \UC ): [T] \mapsto [F(T)].
\]
In particular, the BGG equivalence $\BGG: \stgrmod{E} \xrightarrow{\sim} D^b( \Coh \Pro{V} $ induces an isomorphism
\[
  K_0( \stgrmod{E} ) \simeq K_0( D^b( \Coh \Pro{V}  ) )
\]
which we are now going to investigate.
First, the Grothendieck group of $K_0( \stgrmod{E} )$ is given as the quotient
$K_0( E\grmod )/ \langle [P] \mid \text{$P$ is projective} \rangle$.
Such a description is actually valid for every stabilization of an exact category
(see \cite[Corollary in Section 9.5]{Hap87}).
Now, since every object in $E\grmod$ admits a finite filtration with
composition factors in $\{k(i) \mid i \in \Z\}$,
the group $K_0( E\grmod )$ is generated by the classes $[k(i)]$.
Moreover, there are no relations between these generators because
$[M] = \sum_{i \in \Z} \dim_k( M_i )[ k(-i) ]$
encodes the Hilbert function of $M \in E\grmod$,
which is an additive invariant. Thus, we get an isomorphism
\begin{equation}\label{equation:k0stable}
 K_0( \stgrmod{E} ) \simeq \Z[t,t^{-1}]/ \langle \sum_{i=0}^{n+1}t^{i}\binom{n}{i}\rangle_{\Z[t,t^{-1}]} 
\end{equation}
that identifies $\underline{k(-i)}$ with the indeterminate $t^i$.
\begin{lemma}\label{lemma:K_0bgg}
 Under the identifications \eqref{equation:k0coh} and \eqref{equation:k0stable},
 the isomorphism of Grothendieck groups
 \[
   K_0( \stgrmod{E} ) \simeq K_0( D^b( \Coh \Pro{V}  ) )
 \]
 induced by the BGG correspondence is given by
\begin{align*}
\Z[t,t^{-1}]/ \langle \sum_{i=0}^{n+1}t^{i}\binom{n}{i}\rangle_{\Z[t,t^{-1}]} &\longrightarrow \Z[x,x^{-1}]/ \langle \sum_{i=0}^{n+1}(-1)^ix^{-i}\binom{n}{i}\rangle_{\Z[x,x^{-1}]}\\
t^i &\longmapsto (-1)^ix^{-i}
\end{align*}
\end{lemma}
\begin{proof}
 Via the BGG correspondence the object $\underline{k(-i)}$ is sent to $\mathcal{O}_{\Pro V}(i)[-i]$. 
 Since $[T[i]] = (-1)^i[T] \in K_0( \TC )$ for
 all $T$ in a triangulated category $\TC$, the claim follows.
\end{proof}

\begin{corollary}\label{corollary:computing_rank_and_chern}
 Let $\underline{M} \in \stgrmod{E}$.
 Then 
 \[
  \rank( \BGG(\underline{M}) ) = \sum_{i \in \Z}(-1)^i\dim_k(M_i) \in \Z
 \]
 and
 \[
  \chern( \BGG(\underline{M}) ) = \prod_{i \in \Z}(1+ih)^{(-1)^i\dim_k(M_{-i})} \in (\Z[h]/\langle h^{n+1}\rangle)^{\times}.
 \]
\end{corollary}
\begin{proof}
 Immediate from Example \ref{example:K_0coh}
 and Lemma \ref{lemma:K_0bgg}.
\end{proof}

We can now state and prove our necessary condition
for being a vector bundle up to translation and twist that
can be read off from the Hilbert series.

\begin{corollary}[Necessary condition]\label{corollary:necessary_condition}
 Let $\underline{M} \in \stgrmod{E}$ such that $\BGG(\underline{M})$ is
 isomorphic to a vector bundle up to translation and twist.
 Define $r := \sum_{i \in \Z}(-1)^i\dim_k(M_i)$.
 We set $e := 1$ if $r < 0$, and $e := 0$ otherwise.
 Then the degree of
 \[
  \prod_{i \in \Z}(1+ih)^{(-1)^{i+e}\dim_k(M_{-i})} = 1 + c_1 h + c_2 h^2 + \dots + c_n h^n \in (\Z[h]/\langle h^{n+1}\rangle)^{\times}
 \]
 is lower or equal to $\abs{r}$, where $c_1, \dots, c_n \in \Z$.
\end{corollary}
\begin{proof}
 Let $\mathcal{E}$ be a vector bundle and $j \in \Z$ 
 such that $\BGG(\underline{M}) \cong \mathcal{E}[j]$.
 Then $\rank( \mathcal{E} ) = |r|$ 
 and $\chern( \mathcal{E} ) = \chern( \mathcal{E}[j+e] ) = \prod_{i \in \Z}(1+ih)^{(-1)^{i+e}\dim_k(M_{-i})}$
 by Corollary \ref{corollary:computing_rank_and_chern}.
 Now, the claim follows from the fact that the degree of the Chern polynomial of a vector bundle is lower 
 or equal to the rank.
\end{proof}

\subsection{Equivariant BGG correspondence}\label{subsection:equivariant}
Let $G$ be a finite group and $k$ be a splitting field for $G$.
Assume $G$ acts linearly on $V$. Then this action extends to an action on $E$
by $\Z$-graded algebra automorphisms.
Now, the BGG correspondence can be lifted to a $G$-equivariant setup
as follows:
let $(E\rtimes G)\grmod$ denote the category of finitely generated $G$-equivariant $\Z$-graded $E$-modules,
i.e., 
the objects are $\Z$-graded $k$-vector spaces $M$ equipped with a graded $E$-module structure and a $G$-action 
compatible with the grading
such that the multiplication map $E \otimes M \rightarrow M$ is $G$-equivariant,
and the morphisms are graded linear maps compatible with both the $E$-action and $G$-action.
Then this category is a Frobenius category by the following lemma.

\begin{lemma}
 \mbox{}
   \begin{enumerate}
    \item Every projective object in $(E\rtimes G)\grmod$ is isomorphic to
          an object of the form 
          $E \otimes U$ for some $\Z$-graded finite dimensional $G$-module $U$ and
          conversely, every object of this form is projective.
    \item The classes of projective and injective objects in $(E\rtimes G)\grmod$ coincide.
    \item Every object in $(E\rtimes G)\grmod$ has a projective and an injective resolution.
   \end{enumerate}
\end{lemma}
\begin{proof}
 As a graded $E$-module, $E \otimes U$ is projective.
 Since $\Hom_{(E\rtimes G)\grmod}( E \otimes U, - )$
 can be computed as the composition of the exact functor $\Hom_{E\grmod}( E \otimes U, - )$
 and the exact functor of invariants, 
 $E \otimes U$ is also projective as an object in $(E\rtimes G)\grmod$.
 
 Given any $M \in (E\rtimes G)\grmod$, the epimorphism $\pi: M \rightarrow M/\rad(M)$,
 where $\rad(M)$ denotes the graded radical of $M$ as an $E$-module,
 is a $G$-equivariant morphism. Since $k$ is a splitting field of $G$,
 $\pi$ admits a section on the level of graded $G$-modules, i.e., a graded $G$-module morphism $\sigma: M/\rad(M) \rightarrow M$
 such that $\pi \circ \sigma = \id$.
 Then, the morphism $E \otimes \im( \sigma ) \rightarrow M: e \otimes m \mapsto em$
 is a projective cover of $M$ in $(E\rtimes G)\grmod$, since it is also a projective cover in $E\grmod$.
 If $M$ is projective, then it is its own projective cover and therefore
 isomorphic to an object of the form $E \otimes U$
 by the uniqueness of projective covers.
 Using the duality $\Hom_k(-,k)$, we get all the corresponding statements for injective modules.
 Since $\Hom_k(-,k)$ respects the class of objects of the form $E \otimes U$,
 projectives and injectives coincide.
\end{proof}

By modding out the class of projective-injective objects,
we can form the stable category $\stgrmod{E\rtimes G}$
that equivalently may be described as $\CTate( (E\rtimes G)\grmod )$ (see Theorem \ref{theorem:stable_cat}).
It is shown in \cite[Theorem 9.1.2]{FloystadBGG} that
\[D^b( \Coh( \Pro V \rtimes G ) ) \simeq \CTate( (E\rtimes G)\grmod )\]
where $D^b( \Coh( \Pro V \rtimes G ) )$ is the bounded derived category
of $G$-equivariant coherent sheaves on $\Pro V$.
Combining this result with $\CTate( (E\rtimes G)\grmod ) \simeq \stgrmod{E\rtimes G}$,
we get our desired $G$-equivariant BGG correspondence
\[
 D^b( \Coh( \Pro V \rtimes G ) ) \simeq \stgrmod{E\rtimes G}.
\]

\begin{remark}
It is hard to find an explicit form of the equivariant BGG correspondence in the literature. The reason for this
might be that it seems very likely that one can state and prove the BGG correspondence
internal to an appropriate class of tensor categories without using any new argument,
but simply by using the language of tensor categories.
\end{remark}

\subsection{Adaption of the construction strategy to the equivariant setup}\label{subsection:adaption}
Let $T$, $U$ be finite dimensional $\Z$-graded \emph{irreducible} $G$-modules concentrated in degrees $t$, $u$, respectively.
If $t \not= u$, then any monomorphism $\phi: (\bigwedge^{n+1}W \otimes_k T) \longrightarrow E \otimes_k (\bigwedge^{n+1}W \otimes_k U)$
of graded $G$-modules defines
a morphism of $G$-equivariant graded $E$-modules
\[
 \widehat{\phi}: E \otimes_k (\bigwedge^{n+1}W \otimes_k T) \longrightarrow E \otimes_k (\bigwedge^{n+1}W \otimes_k U)
\]
that fits into a minimal Tate sequence. In fact,
we have already seen in Subsection \ref{subsection:construction_strategy} that
the natural monomorphism $\kernel( \widehat{\phi} ) \hookrightarrow E \otimes_k (\bigwedge^{n+1}W \otimes_k T)$
is the injective hull of a reduced $E$-module.
To see that $\widehat{\phi}$ fits into a minimal Tate sequence
it suffices to show that $\cokernel(\widehat{\phi})$ is a reduced $E$-module as well.
But this follows from the fact that
$\soc( E \otimes_k (\bigwedge^{n+1}W \otimes_k U) ) \subseteq \image(\widehat{\phi})$
which holds due to the irreducibility of $U$.
So, in addition to equality \eqref{equation:h0}, we also get
\begin{equation}\label{equation:hm1}
H^i\left(\BGG( \kernel( \widehat{\phi} ) )(-i+2) \right)
=
\left\{ 
    \begin{array}{cc}
                 U, & \text{for~}(-i+2) = u,\\
                 0, & \text{else,}
    \end{array} 
   \right. 
\end{equation}
for all $i \in \Z$.

We summarize the context of our construction strategy
in the following definition.

\begin{definition}
 Let $G$ be a finite group, $k$ a splitting field for $G$,
 and let
 $V, T, U$ be finite dimensional $\Z$-graded irreducible $G$-modules 
 concentrated in degrees $-1$, $t$, $u$, respectively, such that $t \not= u$.
 We define the associated \textbf{BGG space} as the subset
 of the projective space
 \[\Pro\big( \Hom_{kG\grmod}( (\bigwedge^{n+1}W \otimes_k T), E \otimes_k (\bigwedge^{n+1}W \otimes_k U) ) \big)\]
 consisting of those points $[\phi]$ such that $\phi$ is a monomorphism,
 where $E := \bigwedge V$ and $W := \Hom_k( V, k )$.
 We denote this subset by $\BGG_V( T, U )$.
\end{definition}

Points in $\BGG_V( T, U )$ give rise to
objects in $D^b( \Coh( \Pro V \rtimes G ) )$ with hypercohomology groups
given by the equations \eqref{equation:h0} and \eqref{equation:hm1}.

\begin{definition}
 Let $\BGG_V(T,U)$ be a BGG space consisting of exactly one point $[\phi]$.
 Then we call the $G$-equivariant object $\BGG( \kernel( \widehat{\phi} ) ) \in D^b( \Coh( \Pro V \rtimes G ) )$
 \textbf{strongly determined}.
\end{definition}

\begin{remark}\label{remark:character_theroy_BGG_space}
 It is easy to see that $\BGG_V(T,U)$ is a singleton
 if and only if
 \[\Hom_{kG\grmod}\big( (\bigwedge^{n+1}W \otimes_k T), E \otimes_k (\bigwedge^{n+1}W \otimes_k U) \big) \]
 is of dimension $1$,
 and this dimension can be computed as the scalar product 
 of the dual character associated to
 $(\bigwedge^{n+1}W \otimes_k T)$ and the character associated to the $t+n+1$ degree part of
 $E \otimes_k (\bigwedge^{n+1}W \otimes_k U)$.
\end{remark}

\section{Strongly determined $A_5$-equivariant vector bundles}\label{section:A5}
In this section we work with $k = \C$.
Let $A_5 = \langle (1~2~3~4~5), (3~4~5) \rangle \leq S_5$
denote the alternating group on $5$ points.
The goal of this section is to find
all strongly determined $A_5$-equivariant vector bundles
on $\Pro V$, where $V$ is the unique irreducible $A_5$-representation of degree $5$
(and thus $\Pro V \simeq \Pro^4$).
We start by recalling the character table of $A_5$:
\[
 \begin{array}{r|r|r|r|r|r|}   
      & () & (1~2)(3~ 4) & (1~ 2~ 3) & (1~ 2~ 3~ 4~ 5) & (1~ 2~ 3~ 5~ 4) \\
      \hline
      & & & & &\\[-1em]
      \chi^1  &   1  &1  &1  &1  &1 \\
      \chi^3  &   3 &-1  &\cdot  &\sigma\Phi &\Phi \\
      \sigma\chi^3  &   3 &-1  &\cdot &\Phi  &\sigma\Phi \\
      \chi^4  &   4  &\cdot  &1 &-1 &-1 \\
      \chi^5  &   5  &1 &-1  &\cdot  &\cdot \\
 \end{array}
\]
We have $\Phi = \frac{1+\sqrt{5}}{2}$ with its Galois conjugate $\sigma\Phi = \frac{1-\sqrt{5}}{2}$,
where $\sigma$ is the non-trivial element in the Galois group of $\Q(\Phi)$ over $\Q$.
By abuse of notation we will use the same symbol for a character $\chi$ and 
a representation affording $\chi$. If we want to consider 
a representation $\chi$ as a $\Z$-graded object concentrated in (internal) degree $d \in \Z$,
we decorate its symbol with an index: $\chi_d$.
Furthermore, we will highlight the individual steps of our
construction procedure by putting a box around them.

\begin{center}
  \begin{minipage}{\textwidth}
  \fbox{\textbf{Step 1:} Set $V := \chi^5_{-1}$.} 
  \end{minipage}
\end{center}

Note that the $k$-dual $W := \Hom( V, k ) =\chi^5_1$.
Furthermore, $\bigwedge^5 W = \chi^1_5$.
In particular, tensoring a graded representation with $\bigwedge^5 W$ simply shifts its degree by $5$.
We denote the set of irreducible characters of $A_5$ by $X := \{ \chi^1, \chi^3, \sigma\chi^3, \chi^4, \chi^5 \}$.

\begin{center}
  \begin{minipage}{\textwidth}
  \fbox{\textbf{Step 2:} Find all $i \in \{-1,-2,-3,-4\}$ and $\varphi,\psi \in X$ such that $\BGG_V( \varphi_i, \psi_0 )$ is a singleton.}
  \end{minipage}
\end{center}

This step can be performed by using character theory (see Remark \ref{remark:character_theroy_BGG_space}).
The cases $i>0$ and $i<-5$ lead to empty BGG spaces.
In the case $i = -5$, every point $\phi \in [\BGG_V( \varphi_i, \psi_0 )]$
gives us (see Subsection \ref{subsection:adaption} for the definition of $\widehat{\phi}$)
\[
 \kernel( \widehat{\phi} ) \cong \big( E \otimes_k (\bigwedge^{5}W \otimes_k \psi_{-5}) \big)_{< 0 }
\]
which is the first syzygy object in a resolution of $\bigwedge^{5}W \otimes_k \psi_{-5} \cong \psi_0$
(regarded as a trivial $E$-module).
Thus, the case $i=-5$ will always give us the $A_5$-equivariant 
vector bundle $\mathcal{O}_{\Pro V} \otimes \psi_0(-1)$.

\begin{center}
  \begin{minipage}{\textwidth}
  \fbox{
    \begin{minipage}{\textwidth}
      \textbf{Step 3:} For $(i,\varphi,\psi)$ found in Step $2$, let $[\phi] \in \BGG_V( \varphi_i, \psi_0 )$.
      Compute the morphism $\widehat{\phi}: E \otimes_k (\bigwedge^{5}W \otimes_k \phi_i) \longrightarrow E \otimes_k (\bigwedge^{5}W \otimes_k \psi_0)$
      (see Subsection \ref{subsection:adaption}).
      Determine the Hilbert series of $\kernel( \widehat{\phi} )$
      and check the necessary condition described in Corollary \ref{corollary:necessary_condition}.
    \end{minipage}
  }
  \end{minipage}
\end{center}

If the necessary condition does not hold, then repeat Step $3$ with the next triple $(i,\varphi,\psi)$
found in Step $2$ (if there is any such triple left).

\begin{remark}
 It is tempting to think that the computation of the Hilbert series of 
 $\kernel( \widehat{\phi} )$ might work without the computation of the morphism $\widehat{\phi}$ itself,
 but simply by looking at the characters of $E \otimes_k (\bigwedge^{5}W \otimes_k \phi_i)$ and $E \otimes_k (\bigwedge^{5}W \otimes_k \psi_0)$.
 For example, let $\phi_i = \chi_{-2}^4$ and $\psi_0 = \chi_{0}^1$.
 We enlist the (internal) degree-wise decompositions into characters in the following table: 
  \[
  \begin{array}{r|r|r|r}
   \deg & \kernel( \widehat{\phi} )  & E \otimes_k (\bigwedge^{5}W \otimes_k \phi_i) & E \otimes_k (\bigwedge^{5}W \otimes_k \psi_0) \\
  \hline
  &&& \\[-1em] 
 -2 & \chi^{4}  & \chi^{4}  &   \\ 
&&& \\[-1em] 
-1 & \chi^{3} + \sigma\chi^{3} + \chi^{4} + 2\chi^{5}  & \chi^{3} + \sigma\chi^{3} + \chi^{4} + 2\chi^{5}  &   \\ 
&&& \\[-1em] 
0 & 2\chi^{3} + 2\sigma\chi^{3} + 3\chi^{4} + 3\chi^{5}  & \chi^{1} + 2\chi^{3} + 2\sigma\chi^{3} + 3\chi^{4} + 3\chi^{5}  & \chi^{1}   \\ 
&&& \\[-1em] 
1 & \chi^{1} + 2\chi^{3} + 2\sigma\chi^{3} + 3\chi^{4} + 2\chi^{5}  & \chi^{1} + 2\chi^{3} + 2\sigma\chi^{3} + 3\chi^{4} + 3\chi^{5}  & \chi^{5}   \\ 
&&& \\[-1em] 
2 & 2\chi^{5}  & \chi^{3} + \sigma\chi^{3} + \chi^{4} + 2\chi^{5}  & \chi^{3} + \sigma\chi^{3} + \chi^{4}   \\ 
&&& \\[-1em] 
3 &  & \chi^{4}  & \chi^{3} + \sigma\chi^{3} + \chi^{4}   \\ 
&&& \\[-1em] 
4 &  &  & \chi^{5}   \\ 
&&& \\[-1em] 
5 &  &  & \chi^{1}   \\ 
  \end{array}
 \]
 We can see that the second column is given 
 by subtracting the forth column from the third column
 and keeping only those terms with positive coefficient.
 Interestingly, in the case that we are interested in (our group equals $A_5$ and $V = \chi^5_{-1}$),
 the decomposition of $\kernel( \widehat{\phi} )$ into characters can always be computed by such a subtraction
 (a fact that we could only check a posteriori, i.e., after having already computed $\kernel( \widehat{\phi} )$).
 However, this is only a coincidence and not true in general. 
 Counterexamples exist in the case where our group
 equals the symmetric group on $5$ points, e.g.,
 if $\chi$ denotes the irreducible character of degree $5$ 
 that maps the conjugacy class consisting of order $6$ elements to $1$,
 and if $\sgn$ denotes the signum, then $V = \phi_i = \chi_{-1}$ and $\psi_0 = (\sgn \chi)_0$ provides a counterexample.
\end{remark}

We give an example where the necessary condition checked in Step 3 does not hold.

\begin{example}
 Let $[\phi] \in \BGG_V( \chi_{-2}^4, \chi_{0}^1 )$ be the unique point
 and $\mathcal{F} := \BGG( \kernel( \widehat{\phi} ) )$.
 Using the methods described in Subsection \ref{subsection:K0}, we can compute 
 \begin{align*}
  \rank( \mathcal{F} ) &= 2, \\
  \chern( \mathcal{F} ) &= 1-h+3h^2+5h^3+10h^4
 \end{align*}
 up to translation and twist.
 The necessary condition described in Corollary \ref{corollary:necessary_condition}
 does not hold and thus, $\mathcal{F}$ cannot correspond to a vector bundle up to translation and twist.
 A finite excerpt of its cohomology table looks as follows:
 \settowidth{\mycolwd}{$140$}
\[
 \begin{array}{c*{11}{C{\mycolwd}}}   
 140\ct  &66  &25\ct   &6   &\ct   &   &\ct   &   &\ct   &   &\ct\\
  30\ct  &25  &20\ct  &15  &10\ct   &4   &\ct   &   &\ct   &   &\ct\\
   \ct   &   &\ct   &   &\ct   &   &1\ct   &   &\ct   &   &\ct\\
   \ct   &   &\ct   &   &\ct   &   &\ct   &   &\ct   &   &\ct\\
   \ct   &   &\ct   &   &\ct   &   &\ct   &6  &25\ct  &66  &140\ct
 \end{array}
\]
\end{example}

If the necessary condition does hold, then we can proceed further.

\begin{center}
  \begin{minipage}{\textwidth}
  \fbox{
    \begin{minipage}{\textwidth}
      \textbf{Step 4:} Find a finite excerpt of the cohomology table of $\BGG( \kernel( \widehat{\phi} ) )$
      by computing a few steps of a minimal injective and projective resolution of $\kernel( \widehat{\phi} )$.
      Check if the criterion described in Lemma \ref{lemma:sufficient_condition} holds up to translation and twist.
    \end{minipage}
  }
  \end{minipage}
\end{center}

If $\kernel( \widehat{\phi} )$ passed Step $3$ but did not correspond to a vector bundle up to translation and twist,
then we would get stuck in Step $4$ since we cannot disprove the existence of
the two sparse columns required by Lemma \ref{lemma:sufficient_condition}
by only computing finite excerpts of the cohomology table.
On the other hand, if $\kernel( \widehat{\phi} )$ does correspond to a vector bundle,
it is furthermore unclear how many steps are actually needed in the resolutions
for finding two sparse columns.
Nevertheless, if Step $4$ succeeds, then we have found
a strongly determined $A_5$-equivariant vector bundle on $\Pro V$.
In order to find all such bundles, we have to go back to Step $3$ and proceed with the next triple $(i,\phi, \psi)$
(if there is any such triple left).

Luckily, in the case that we are examining (our group equals $A_5$ and $V = \chi^5_{-1}$),
every $\kernel( \widehat{\phi} )$ that passes Step $3$ turns out to be a vector bundle.
\begin{example}
 We discuss a case where the vector bundle criterion does hold.
 Let $[\phi] \in \BGG_V( \chi_{-2}^3, \chi_{0}^1 )$ be the unique point
 and $\mathcal{F} := \BGG( \kernel( \widehat{\phi} ) )$.
 We compute
 \begin{align*}
  \rank( \mathcal{F} ) &= 3, \\
  \chern( \mathcal{F} ) &= 1+2h^2+2h^3
 \end{align*}
 up to translation and twist.
 The necessary condition described in Corollary \ref{corollary:necessary_condition}
 holds, and so we compute a finite excerpt of the cohomology table:
 \settowidth{\mycolwd}{$175$}
\[
 \begin{array}{c*{11}{C{\mycolwd}}}   
 162\ct&  70&  21\ct&  &  \ct&  &  \ct&  &  \ct&   \\
  \ct&  &  \ct&   5&   3\ct&  &  \ct&  &  \ct&   \\
  \ct&  &  \ct&  &  \ct&   1&  \ct&  &  \ct&   \\
  \ct&  &  \ct&  &  \ct&  &  \ct&  &  \ct&   \\
  \ct&  &  \ct&  &  \ct&  &   7\ct&  30&  81\ct& 175\\
 \end{array}
\]
 The criterion of Lemma \ref{lemma:sufficient_condition} also holds,
 thus, $\mathcal{F}$ is (up to translation and twist) a strongly determined $A_5$-equivariant
 vector bundle.
\end{example}

For stating the main theorem we need
the following definitions from Boij-S\"oderberg theory,
in which vector bundles with supernatural cohomology play a crucial role
(for a survey of Boij-S\"oderberg theory, see \cite{FloystadBSSurvey}).

\begin{definition}
 For $s \in \N_0$, a \bfindex{root sequence of length $s$} is a sequence of strictly decreasing integers
 \[
  \mathbf{z}: z_1 > z_2 > \dots > z_s.
 \]
 For our convenience, we will expand every such sequence by $z_i := \infty$ for $i < 1$ and $z_{i} := -\infty$ for $i > s$.
 We set its associated \textbf{Hilbert polynomial} as
 \[
  \HPo^{\mathbf{z}}(t) := \frac{1}{s!}\prod_{i=1}^s(t-z_i)
 \]
 and its associated \textbf{cohomology table} as
 \[
  \gamma^{\mathbf{z}}(p,q) := \left\{ \begin{array}{cc}
                 |\HPo^{\mathbf{z}}(p)|, & \text{for~} z_q > p > z_{q+1},  \\
                 0, & \text{else.}
    \end{array} \right.
 \]
\end{definition}

\begin{definition}
 Let $\mathbf{z}$ be a root sequence.
 A vector bundle $\mathcal{F} \in \Coh \Pro^n $ has \bfindex{supernatural cohomology of type $\mathbf{z}$}
 if the cohomology table of $\mathcal{F}$ is given by 
 \[
 \CH^{q}( \mathcal{F}(p) ) = \rank( \mathcal{F} ) \cdot \gamma^{\mathbf{z}}(p,q).
 \]
\end{definition}

Now, we summarize the findings of our computations that were performed using the software
packages described in the Appendix \ref{appendixa} in the following theorem.

\begin{theorem}\label{theorem:A5_main_theorem}
 Let $V$ be the unique irreducible $A_5$-representation of dimension $5$.
 Then all strongly determined $A_5$-equivariant vector bundles
 on $\Pro V$ are supernatural.
 All $i \in \{-1,-2,-3,-4\}$, and $\phi, \psi \in \{ \chi^1, \chi^3, \sigma\chi^3, \chi^4, \chi^5 \}$
 such that $\BGG_V( \phi_i, \psi_0 )$ is a singleton
 whose element corresponds to a vector bundle $\mathcal{F}$ (up to translation and twist)
 can be found in the following table, 
 together with the type $\mathbf{z}$ of $\mathcal{F}$:
 \[
  \begin{array}{c|c|c|c|c}
  \psi & i & \phi & \mathbf{z} & \rank\\
  \hline
  & & & & \\[-1em]
  \chi^1 & -4 & \chi^5 & [ 0, -2, -3, -4 ] & 4\\ & & & & \\[-1em] 
 & -2 & \chi^{3} & [ 0, -1, -3, -6 ] & 3\\ & & & & \\[-1em] 
 & -2 & \sigma\chi^{3} & [ 0, -1, -3, -6 ] & 3\\ & & & & \\[-1em] 
 & -1 & \chi^5 & [ 0, -1, -2, -3 ] & 1\\ & & & & \\[-1em] 
\chi^{3} & -3 & \chi^{3} & [ 0, -2, -3, -5 ] & 9\\ & & & & \\[-1em] 
 & -3 & \sigma\chi^{3} & [ 0, -2, -3, -5 ] & 9\\ & & & & \\[-1em] 
 & -2 & \chi^1 & [ 0, -3, -5, -6 ] & 3\\ & & & & \\[-1em] 
& -1 & \chi^{3} & [ 0, -1, -4, -5 ] & 6\\ & & & & \\[-1em] 
& -1 & \sigma\chi^{3} & [ 0, -1, -4, -5 ] & 6\\ & & & & \\[-1em] 
& -1 & \chi^5 & [ 0, -1, -3, -6 ] & 3\\ & & & & \\[-1em] 
\sigma\chi^{3} & -3 & \chi^{3} & [ 0, -2, -3, -5 ] & 9\\ & & & & \\[-1em] 
& -3 & \sigma\chi^{3} & [ 0, -2, -3, -5 ] & 9\\ & & & & \\[-1em] 
& -2 & \chi^1 & [ 0, -3, -5, -6 ] & 3\\ & & & & \\[-1em] 
& -1 & \chi^{3} & [ 0, -1, -4, -5 ] & 6\\ & & & & \\[-1em] 
& -1 & \sigma\chi^{3} & [ 0, -1, -4, -5 ] & 6\\ & & & & \\[-1em] 
& -1 & \chi^5 & [ 0, -1, -3, -6 ] & 3\\ & & & & \\[-1em] 
\chi^4 & -1 & \chi^4 & [ 0, -1, -4, -5 ] & 8\\ & & & & \\[-1em] 
\chi^5 & -4 & \chi^1 & [ 0, -1, -2, -4 ] & 4\\ & & & & \\[-1em] 
& -1 & \chi^1 & [ 0, -1, -2, -3 ] & 1\\ & & & & \\[-1em] 
& -1 & \chi^{3} & [ 0, -3, -5, -6 ] & 3\\ & & & & \\[-1em] 
 & -1 & \sigma\chi^{3} & [ 0, -3, -5, -6 ] & 3
  \end{array}
 \]
\end{theorem}

\begin{remark}
 The Horrocks-Mumford bundle \cite{HM73}
 is also a strongly determined equivariant vector bundle
 for its group of symmetries of order $15000$,
 but it does not have supernatural cohomology.
\end{remark}

\begin{appendix}

\section{Software packages}\label{appendixa}

The main tool for our computations is the \CapPkg project (Categories, Algorithms, Programming),
which is a collection of software packages\footnote{\CapPkg's core system is part of the current \GAP release (as of October 2017).} 
written in \GAP. The \CapPkg project supports the programmer
in the implementation of specific instances of categories,
like in our case the category $(E \rtimes G)\grmod$ of finitely generated $G$-equivariant
$\Z$-graded modules over the exterior algebra $E$.

The \GAP packages which enable us to work with the category $(E \rtimes G)\grmod$
are called \texttt{InternalExteriorAlgebraForCAP} 
and \texttt{GroupRepresentationsForCAP}, and they are freely available on GitHub.\footnote{They are part
of the \texttt{CAP\_project} repository: \href{https://github.com/sebastianpos/CAP_project}{\url{https://github.com/sebastianpos/CAP_project}}.}
Up till now an installation of Magma \cite{magma} is also needed for a proper usage of our packages,
since Magma provides fast linear algebra over number fields.
It is planned to resolve this dependency by switching to open software alternatives like Nemo \cite{nemo}.

\section{Sketch: computing with equivariant modules}\label{appendix}

We briefly explain how computations within the category
$(E \rtimes G)\grmod$ become feasible on the computer.
To simplify our explanation, we will drop the $\Z$-grading,
but it is straightforward to add it to our setup in the end.
The main ideas for computing in $(E \rtimes G)\ModLeft$ are
\begin{enumerate}
 \item to reduce all computations to linear algebra,
 \item to keep the resulting matrices small by exploiting $G$-equivariance.
\end{enumerate}
The first idea is easy to implement since every finitely generated $G$-equivariant $E$-module is in particular a finite dimensional vector space.
But the resulting matrices describing the $G$-equivariant and $E$-linear operators
within injective/projective resolutions quickly become very big.
A reduction of this growth can be achieved by exploiting the $G$-equivariance using category theory
in a way that we are now going to explain.

If our ground field $k$ is a splitting field for our finite group $G$,
then, due to Maschke's theorem, the category of finite dimensional $k$-representations of $G$
(denoted by $\Rep(G)$) is equivalent as an abelian category to finitely many copies
of the category of finite dimensional $k$-vector spaces (denoted by $\kvec$),
where the copies are indexed by the irreducible representations $\Irr(G)$ of $G$:
\begin{equation}\label{equation:rep_equiv}
 \Rep(G) \simeq \bigoplus_{\chi \in \Irr(G)} \kvec. 
\end{equation}
A morphism $\alpha$ between two representations $V$ and $W$ in $\Rep(G)$
is modeled by a matrix of dimension $\dim(V) \times \dim(W)$
which, under the equivalence \eqref{equation:rep_equiv}, 
is transformed into an $\Irr(G)$-indexed list of matrices
whose $\chi$-th entry is an $m_\chi \times n_\chi$ matrix,
where $m_\chi$ and $n_\chi$ are the multiplicities of $\chi$
in $V$ and $W$, respectively. Since this list of matrices
may contain much less data then our original matrix $\alpha$,
we prefer computing in $\bigoplus_{\chi \in \Irr(G)} \kvec$ to computing in $\Rep(G)$.
But how can we model the category
$(E \rtimes G)\grmod$ in terms of $\bigoplus_{\chi \in \Irr(G)} \kvec$?
The answer lies in the concept of \emph{internalization}.

Roughly speaking, internalization means reformulating a (maybe classical set-theoretic) concept
in the language of categories equipped with some extra structure.
For example, we may speak about (internal) monoids in any category $\CC$ 
with finite products. 
Such an (internal) monoid is then defined as an object $M \in \CC$ equipped with
morphisms $\mu: M \times M \rightarrow M$ and $\eta: 1 \rightarrow M$ (where $1$
denotes the terminal object in $\CC$, which is given by the empty product) such that the diagrams
\begin{center}
          \begin{tikzpicture}[transform shape,mylabel/.style={thick, draw=black, align=center, minimum width=0.5cm, minimum height=0.5cm,fill=white}]
            \coordinate (r) at (3,0);
            \coordinate (d) at (0,-1.75);
            \node (M1) {$M \times (M \times M)$};
            \node (M2) at ($(M1) + 2*(r)$) {$(M \times M) \times M$};
            \node (B1) at ($(M1) + (d)$) {$M \times M$};
            \node (B2) at ($(M2) + (d)$) {$M \times M$};
            \node (B3) at ($(B1) + (r)$) {$M$};
            \draw[->,thick] (M1) --node[above]{$\simeq$} (M2);
            \draw[->,thick] (M1) --node[left]{$M \times \mu$} (B1);
            \draw[->,thick] (M2) --node[right]{$\mu \times M$} (B2);
            \draw[->,thick] (B1) --node[above]{$\mu$} (B3);
            \draw[->,thick] (B2) --node[above]{$\mu$} (B3);
          \end{tikzpicture}
  \end{center}
  and
  \begin{center}
          \begin{tikzpicture}[transform shape,mylabel/.style={thick, draw=black, align=center, minimum width=0.5cm, minimum height=0.5cm,fill=white}]
            \coordinate (r) at (3,0);
            \coordinate (d) at (0,-2);
            \node (M1) {$1 \times M$};
            \node (M2) at ($(M1) + (r)$) {$M \times M$};
            \node (M3) at ($(M2) + (r)$) {$M \times 1$};
            \node (B1) at ($(M2) + (d)$) {$M$};
            \draw[->,thick] (M1) --node[above]{$\eta \times M$} (M2);
            \draw[->,thick] (M3) --node[above]{$M \times \eta$} (M2);
            \draw[->,thick] (M1) --node[left, inner sep = 0.5em]{$\simeq$} (B1);
            \draw[->,thick] (M3) --node[right, inner sep = 0.5em]{$\simeq$} (B1);
            \draw[->,thick] (M2) --node[left, inner sep = 0.2em]{$\mu$} (B1);
          \end{tikzpicture}
  \end{center}
commute. If $\CC$ is the category of sets, then $M$ is a classical monoid
with multiplication map $\mu$ and unit given by the unique element in the image of $\eta$.
The strength of this definition lies in the fact that we may vary $\CC$. For example,
an (internal) monoid in the category of topological spaces is a topological monoid.
An (internal) monoid in the category of groups is an abelian group.

The diagrammatic definition of an (internal) monoid even makes sense in contexts
where we do not necessarily have products,
but simply a bifunctor $\otimes: \CC \times \CC \rightarrow \CC$,
a distinguished object $1 \in \CC$, associator isomorphisms
$A \otimes (B \otimes C) \simeq (A \otimes B) \otimes C$ and unitor isomorphisms $1 \otimes A \simeq A \otimes 1 \simeq A$ natural in $A,B,C \in \CC$
satisfying certain relations. Such categories are called monoidal categories (see \cite{MLCWM} for a proper definition).
For example, $\Rep(G)$ together with the usual tensor product $\otimes_k$ of representations is a monoidal category.
An (internal) monoid in $(\Rep(G), \otimes_k)$ is the same as a finite dimensional $G$-equivariant algebra
(like our exterior algebra $E = \bigwedge V$ in Subsection \ref{subsection:equivariant}).

Similar to the concept of a monoid, we can also internalize
the concept of a left action of a monoid $(M,\mu,\eta)$ in any monoidal category $(\CC,\otimes)$:
an (internal) left action is simply an object $V \in \CC$ equipped with a morphism $\mu_V: M \otimes V \rightarrow M$
(encoding the action of $M$ on $V$ from the left) satisfying the obvious conditions.
Again, interpreted in the category of sets, this is nothing but the usual left action
of a monoid on a set. But interpreted in $(\Rep(G), \otimes_k)$
with given (internal) monoid $E$, the category of (internal) left actions of $E$ is
nothing but $(E \rtimes G)\ModLeft$.

Now, the equivalence \eqref{equation:rep_equiv} can be used
to transfer the monoidal structure $\otimes_k$ of 
$\Rep(G)$ to a monoidal structure $\otimes'$ of $\bigoplus_{\chi \in \Irr(G)} \kvec$,
as well as to transfer the (internal) monoid $E$
to an (internal) monoid $E'$.
The upshot is: the category of (internal) left actions of $E'$
is still equivalent to $(E \rtimes G)\ModLeft$,
but it is now built upon the concise data structure that represents
$G$-equivariant maps as lists of small matrices. 
Thus, we have finally found our effective model in which we can compute
with $G$-equivariant $E$-modules.

One last remark concerning the transfer of the monoidal structure from 
$\Rep(G)$ to $\bigoplus_{\chi \in \Irr(G)} \kvec$:
the computationally hardest part is the transfer
of the associator isomorphisms. In our \CapPkg implementation
it involves Dixon's method \cite{DixonRep}
for constructing all irreducible representations of a group $G$
as well as the computation of explicit decomposition
isomorphisms of all pairs of tensor products between irreducible representations.
Luckily, for any given group $G$, such an effort has only to be done once.
We simply can store the resulting associators and reuse them
in future sessions (this applies in particular to the associators
that we computed for $G = A_5$ in order to perform the computations for this paper).

\end{appendix}

\def\cprime{$'$} \def\cprime{$'$} \def\cprime{$'$} \def\cprime{$'$}
  \def\cprime{$'$}
\providecommand{\bysame}{\leavevmode\hbox to3em{\hrulefill}\thinspace}
\providecommand{\MR}{\relax\ifhmode\unskip\space\fi MR }
\providecommand{\MRhref}[2]{%
  \href{http://www.ams.org/mathscinet-getitem?mr=#1}{#2}
}
\providecommand{\href}[2]{#2}

\end{document}